\documentclass[12pt,reqno]{article}

\usepackage[utf8]{inputenc}
\usepackage{amsmath}
\usepackage{amsthm}
\usepackage{amssymb}
\usepackage{float}

\usepackage{placeins}
\usepackage{array}
\usepackage{multirow}
\usepackage{hhline}
\usepackage{verbatim}
\usepackage[letterpaper,left=1.5in,right=1in,top=1in,bottom=1in]{geometry}
\usepackage{amsfonts}
\usepackage{enumerate}
\usepackage{tabu}
\usepackage{mathrsfs}

\usepackage{tikz}
\usepackage{pifont}
\usetikzlibrary{positioning}
\usepackage{ marvosym }
\usepackage{verbatim}
\usepackage[mathscr]{euscript}
\usepackage{tabu}
\usetikzlibrary{arrows}
\usetikzlibrary{shapes,snakes}

\usepackage{setspace,color}
\newtheorem{theorem}{Theorem}
\newtheorem{lemma}[theorem]{Lemma}

\newtheorem{corollary}{Corollary}
\newtheorem{definition}{Definition}

\newcommand\widebar[1]{\mathop{\overline{#1}}}
\newtheorem{question}[theorem]{Question}
\begin{document}

\begin{center}{\large{\bf Diffusion: Quiescence and Perturbation}}

{Todd Mullen$^{1}$, Richard Nowakowski${^1}$, Danielle Cox${^2}$}

\end{center}
\vspace{0.5cm}

 \begin{small}
	\noindent  $^{1}$Department of Mathematics and Statistics, Dalhousie University, Halifax, CANADA\\
	 $^{2}$Department of Mathematics,  Mount Saint Vincent University, Halifax, CANADA.\\
	 %Corresponding Author: danielle.cox@msvu.ca\\
		\end{small}
		\vspace{0.5cm}
		

%\publicationdetails{}{}{}{}{} % We will fill in the dates
%\maketitle

\begin{abstract}
Originally proposed by Duffy et al., Diffusion is a variant of chip-firing in which chips from flow from places of high concentration to places of low concentration. In the variant, Perturbation Diffusion, the first step involves a ``perturbation" in which some number of vertices send chips to each of their respective neighbours even though the rules of Diffusion only permit for chips to be sent from richer vertices to poorer vertices. Perturbation Diffusion allows us to expand our study of Diffusion by asking new questions such as ``Given an initial configuration, which vertices, when perturbed, will return the initial configuration after some number of steps in Diffusion." We give some results in this paper that begin to answer this question in the specific case of every vertex initially having 0 chips. We characterize some of the ways a graph can reach such a state in Perturbation Diffusion before focusing on paths in particular with more specific results. 
\end{abstract}

\section{Introduction}

In this paper, we inspect a model regarding the distribution and transfer of chips in a graph. We define our model to be such that every vertex sends a dollar simultaneously to every neighbouring vertex that is poorer than itself, and that this transfer of chips happens continuously, over and over again, ad infinitum. This model is called Diffusion and was introduced on graphs by \cite{duffy:}. We turn our focus to those particular examples in which the number of chips in the system is divisible by the number of people, and the model will eventually lead to a situation in which every vertex winds up with the same number of chips. We will expand on the work of Duffy et al. by introducing a situation in which a vertex will send chips to a neighbour that is neither poorer nor richer than itself. This slight change to the rules laid out by Duffy et al. will be the catalyst for the results in this paper.

Introduced by \cite{duffy:}, Diffusion is a process defined on a simple finite graph, $G$, in which each vertex is assigned an integer to represent the size of a stack of chips. At each time step, the chips are redistributed via the following rules. If a vertex $v$ is adjacent to a vertex $u$ with fewer chips, $v$ takes a chip from its stack and adds it to the stack of $u$. We say that $v$ \textit{sends} a chip and $u$ \textit{receives} a chip. When a vertex sends a chip, we say it \textit{fires}. An example of Diffusion is provided in Figure~\ref{fig:parex}. In Figure~\ref{fig:parex}, we see at each time step, the vertices of $P_5$ have a stack size. This assignment of stack sizes to the vertices of a graph is referred to as a \textit{configuration}. 

\begin{comment}
Given a graph $G$ and an initial configuration $C_0$, the \textit{configuration sequence} $Seq(C_0)$ is the sequence of configurations that arises as the time increases. 
\end{comment}

In this paper, we define and analyze the Diffusion variant, Perturbation Diffusion. We define Perturbation Diffusion on a graph $G$ as the variant of Diffusion where the first firing (which takes place at step 0 and is referred to as the \textit{initial firing}) is such that for some $H \subseteq V(G)$, for each vertex $v \in H$, $v$ sends a chip to each of its neighbours. After the initial firing, Perturbation Diffusion is identical to Diffusion. 

We refer to the configuration in which every vertex has 0 chips as the \textit{0-configuration}. Note that \cite{duffy:} refer to any configuration in which every stack size is equal as ``fixed." Let $G$ be a graph with the 0-configuration and $H$ be a subset of $V(G)$, a \textit{perturbation} of $H$ is when the vertices of $H$ each send a chip to each of their respective neighbours in $G$. In this case, we call $H$ a \textit{perturbation subset} (an example of a perturbation is shown in Figure~\ref{fig:quantumexample}). Let $Seq(C_0)= (C_0,C_1,C_2, \dots)$ be the configuration sequence on a graph $G$ with initial configuration $C_0$. The positive integer $p$ is a \textit{period length} if $C_t = C_{t+p}$ for all $t \geq N$ for some $N$. In this case, $N$ is a \textit{preperiod length}. For such a value, $N$, if $k \geq N$, then we say that the configuration, $C_k$, is \textit{inside} the period. 
For the purposes of this paper, all references to period length will refer to the \textit{minimum period length $p$} in a given configuration sequence. Also, all references to preperiod length will refer to the \textit{least preperiod length} that yields that minimum period length $p$ in a given configuration sequence. 

Note how Diffusion allows a vertex to go into debt if adjacent to more poorer vertices than it has chips. This is notable because otherwise Diffusion would necessarily exhibit periodic behaviour given that it runs for infinitely many time steps and each vertex would only have finitely many possible stack sizes between 0 and the total number of chips in the model. It is true that Diffusion is eventually periodic, but this result is far from trivial. \cite{long:} showed that given any configuration on any graph, the minimum period length is always either 1 or 2. This result however, provides no insight into the length of pre-periods. So, in attempting to determine the number of steps between a perturbation and the beginning of the period, we will not be able to lean on this previous result.

\cite{duffy:}, showed that Diffusion is such that an addition of some constant $k$, $k \in \mathbb{Z}$, to each stack size will have no effect on determining when and if a chip will move from one vertex to another. 

\begin{comment}
Due to this result, we can see that any attempt to count every configuration which exhibits some behaviour will always return an infinite value because there are infinite integers. So when attempting to count configurations, we will adopt the convention of fixing some stack size at $0$. Clearly, this is negligible with Quantum Diffusion because the initial configuration is defined to be the one in which every stack size is initially $0$.
\end{comment} 
So if one wanted to view Diffusion as a process in which stack sizes are never negative, one would only need to add a sufficient constant $k$, $k \in \mathbb{N}$, to each stack size. \cite{carlotti:} showed how to find such a constant in their paper.

\begin{definition}
	Let $G$ be a graph and let $H$ be a subset of $V(G)$.
	
	\begin{itemize}
		\item We say $H$ is \textbf{$0$-invoking} if a perturbation of $H$ eventually results in a $0$-configuration. 
		\item We say $H$ is \textbf{$0_2$-invoking} if after a perturbation of $H$ and the subsequent firing, the resulting configuration is the 0-configuration.
		\item The \textbf{perturbation quiescent number} of a graph $G$, denoted \textbf{PQ$(G)$}, is the size of the smallest nontrivial $0$-invoking subset of $V(G)$. So, $PQ(G) = min\{|H|: H \neq \emptyset$ is $0$-invoking in $G\}$. 
		\item The \textbf{2-perturbation quiescent number} or \textbf{PQ$_2(G)$} is the size of the smallest nontrivial $0_2$-invoking subset of $V(G)$. So, $PQ_2(G) = min\{|H|: H \neq \emptyset$ is $0_2$-invoking in $G\}$.
	\end{itemize}
\end{definition}

Note that $PQ(G)$ and $PQ_2(G)$ are well-defined because $V(G)$ is itself both a $0$-invoking subset and a $0_2$-invoking subset of $V(G)$. 

\vspace{0.5cm} 

In this paper, we study the following question: Which perturbation subsets will yield a period of length 1? We first characterize all $0_2$-invoking subsets on all graphs in Theorem~\ref{thm:qdominatingiff} and then leave open the problem of determining if there exists a perturbation subset on any graph $G$ that is $0$-invoking and not $0_2$-invoking. We then turn our attention to paths in particular. In Theorem~\ref{thm:qdominatingpath}, we count the number of $0_2$-invoking subsets that exist on $P_n$ for all $n \geq 1$. In Theorem~\ref{thm:qdominationnumber}, we show that $PQ_2(P_n) = \lceil \frac{n}{3} \rceil$, the same as the domination number, for all $n \geq 1$.

\begin{figure}[H]
	\centering
	\begin{tikzpicture}[-,-=stealth', auto,node distance=1.5cm,
	thick,scale=0.6, main node/.style={scale=0.6,circle,draw, minimum size=1cm, font=\sffamily\Large\bfseries}]                   
	\node[main node, fill=black, label={[red]90:0}] (1) 		                                    	        {$H$};
	\node[main node, fill=black, label={[red]90:0}] (2)                      [right of=1]      		        {$H$};
	\node[main node, label={[red]90:0}] (3)                      [right of=2]                    {}; 
	\node[main node, fill=black, label={[red]90:0}] (4)                      [right of=3]        		    {$H$};
	\node[main node, fill=black, label={[red]90:0}] (5)  	                [right of=4]       			    {$H$}; 
	\node[main node, label={[red]90:0}] (6)	                    [right of=5]                    {}; 
	\node[main node, label={[red]90:0}] (7)  	[below=1cm of 1]        	    {};
	\node[main node, label={[red]90:-1}] (8)  [right of=7]       			    {};
	\node[main node, label={[red]90:2}] (9)  	[right of=8]      			    {};
	\node[main node, label={[red]90:-1}] (10) [right of=9]     			    {};
	\node[main node, label={[red]90:-1}] (11) [right of=10]    			    {};
	\node[main node, label={[red]90:1}] (12) 	[right of=11]      		        {};
	
	\draw
	(1) edge (2)	
	(4) edge  (5)
	(10) edge (11)
	;
	
	\draw [->] (2) edge  (3) (4) edge  (3) (5) edge  (6) (7) edge (8) (9) edge (8) (9) edge (10) (12) edge (11);

	\end{tikzpicture} 
	\caption{Perturbation of the subset $H$ (marked with black vertices) of $V(P_6)$ with directed edges depicting the flow of chips: first from vertices of $H$ to vertices of $V(P_6) \setminus H$, and then from richer vertices to poorer vertices}
	\label{fig:quantumexample} 
\end{figure}
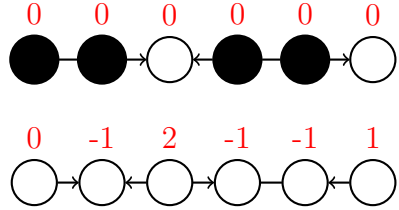 

\FloatBarrier

We now offer more formal definitions regarding Diffusion. The assigned value of a vertex $v$ in a configuration $C$ is its \textit{stack size in $C$} and is denoted $|v|^C$. We omit the superscript when the configuration is clear.
A vertex $v$ is said to be \textit{richer} than another vertex $u$ in configuration $C$ if $|v|^C > |u|^C$. In this instance, $u$ is said to be \textit{poorer} than $v$ in $C$. If $|v|^C<0$, we say $v$ is \textit{in debt in $C$}. In Diffusion, the stack size of a vertex, $v$, at step $t$, is referred to as its \textbf{stack size at time $t$}. If the initial configuration is $C$, then the stack size of $v$ at time $t$ is denoted $|v|_t^C$. This implies that $|v|^C = |v|_0^C$. We omit the superscript when the configuration is clear.

In Diffusion, given a graph $G$ and a configuration $C$ on $G$, to \textit{fire} $C$ is to decrease the stack size of every vertex $v \in V(G)$ by the number of poorer neighbours $v$ has and increase the stack size of $v$ by the number of richer neighbours $v$ has. More formally, for all $v$, let $Z_{-}^{C}(v) = \{u \in N(v) : |v|^{C} > |u|^{C}\}$ and let $Z_{+}^{C}(v) = \{u \in N(v) : |u|^{C} > |v|^{C}\}$. Then, firing results in every vertex $v$ changing from a stack size of $|v|^{C}$ to a stack size of $|v|^{C} + |Z_{+}^{\:C}(v)| - |Z_{-}^{C}(v)|$. Given a set of vertices $A \subseteq V(G)$, the subgraph induced by $A$ will be denoted as \textit{$G|_A$}.

\begin{comment}

Let ${\widebar{Seq(C_0)}}$ be the singleton or ordered pair of configurations contained within the period of a configuration sequence $Seq(C_0)$. If $Seq(C_0)$ has period 2, define the first element of the ordered pair $\widebar{Seq(C_0)}$ to be the one which occurs first in the configuration sequence. Let $C$ be a configuration on a graph $G$. Let $\textit{C+k}$ be the configuration created by adding an integer $k$ to every stack size in the configuration $C$. Two configuration sequences, $Seq(C)$ and $Seq(D)$, are \textit{equivalent} if $\widebar{Seq}(C+k) = \widebar{Seq}(D)$ for some integer $k$. For all configurations $C$ and all integers $k$, we say that $C$ and $C+k$ are \textit{equivalent}.

\end{comment}

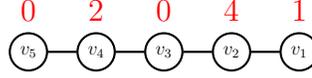
\begin{figure}[H]
	\[
	\begin{tikzpicture}[-,-=stealth', auto,node distance=1.5cm,
	thick,scale=0.6, main node/.style={scale=0.6,circle,draw}]
	
	\node[main node, label={[red]90:0}] (1) 					    {$v_5$};					% Draws a vertex
	\node[main node, label={[red]90:2}] (2)  [right of=1]        {$v_4$};
	\node[main node, label={[red]90:0}] (3)  [right of=2]        {$v_3$};  
	\node[main node, label={[red]90:4}] (4) 	[right of=3]        {$v_2$};					% Draws a vertex
	\node[main node, label={[red]90:1}] (5)  [right of=4]        {$v_1$}; 
	
	\path[every node/.style={font=\sffamily\small}]		%Defines how the edges are drawn
	
	(1) edge node [] {} (2)			
	(2) edge node [] {} (3)
	(3) edge node [] {} (4)
	(4) edge node [] {} (5);

	\end{tikzpicture} \]

\begin{table}[H] 
	\centering
	\begin{tabular}{|c|c|c|c|c|c|}
		\hline
		        & $v_5$ & $v_4$ & $v_3$ & $v_2$ & $v_1$\\
		\hline
		 Step 0 & 0 & 2 & 0 & 4 & 1\\
		\hline
		 Step 1 & 1 & 0 & 2 & 2 & 2\\
		\hline
		 Step 2 & 0 & 2 & 1 & 2 & 2\\
		\hline
		 Step 3 & 1 & 0 & 3 & 1 & 2\\
		\hline
		 Step 4 & 0 & 2 & 1 & 3 & 1\\
		\hline
		 Step 5 & 1 & 0 & 3 & 1 & 2\\
		\hline
		 Step 6 & 0 & 2 & 1 & 3 & 1 \\
		\hline
	\end{tabular}
\end{table}
	\caption{Stack sizes during several steps in a Diffusion process on $P_5$}
	\label{fig:parex}
\end{figure}
\FloatBarrier

\section{$0_2$-invoking}

In this section, we characterize all $0_2$-invoking subsets of a graph $G$. We begin by defining a new concept \textit{complementary component dominance} and then show, with Theorem~\ref{thm:qdominatingiff}, that a subset of vertices in a graph is $0_2$-invoking if and only if it is complementary component dominant. 
\begin{definition}

Given a graph $G$, a subset $H$ of $V(G)$ is \textbf{Complementary Component Dominant} or \textbf{CCD} if both
		
\begin{enumerate}[(i)]

		\item For all adjacent pairs of vertices, $x,y \in H$, the number of neighbours of $x$ in $V(G) \setminus H$ is equal to the number of neighbours of $y$ in $V(G) \setminus H$
		
		and 
		
		\item For all adjacent pairs of vertices, $u,v \in V(G) \setminus H$, the number of neighbours of $u$ in $H$ is equal to the number of neighbours of $v$ in $H$.
		
\end{enumerate}
\end{definition}

Note that this definition implies that if $H$ is complementary component dominant in $G$, then so is $V(G) \setminus H$. 

\begin{theorem}\label{thm:qdominatingiff}
	Let $G$ be a graph with the fixed configuration. A subset $H$ of $V(G)$ is $0_2$-invoking in $G$ if and only if $H$ is CCD.
\end{theorem}

\begin{proof}
	($\Leftarrow$) Let a graph $G$ have the fixed configuration. Suppose $H \subseteq V(G)$ is CCD. In Figure~\ref{fig:q-dominating}, we see $G|_H$ and $G|_{V(G) \setminus H}$ separated into their respective connected components.

	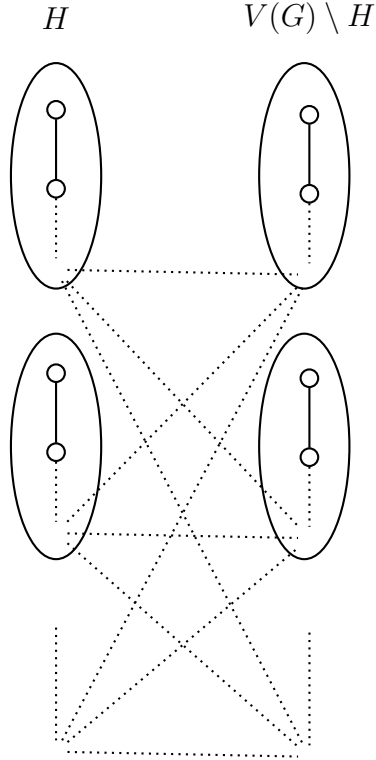
\begin{figure}[H]
		\centering

		\begin{tikzpicture}[-,-=stealth', auto,node distance=1.5cm,
		thick,scale=0.6, main node/.style={scale=0.6,circle,draw}]
		\node[draw=none, fill=none] (1000)                          {$H$}; 
		\node[draw=none, fill=none] (1001)  [right=2cm of 1000]     {$V(G) \setminus H$};         
		\node[main node] (1) 				  [below=0.8cm of 1000]	  {};				
		\node[main node] (2)                [below=0.8cm of 1001]   {};
		\node[main node] (3) 			      [below=0.8cm of 1]	  {};				
		\node[main node] (4)                [below=0.8cm of 2]      {};
		\node[draw=none, fill=none] (1002)  [below=0.8cm of 3]        {};  
		\node[draw=none, fill=none] (1003)  [below=0.8cm of 4]        {};
		\node[main node] (5) 				  [below=1.1cm of 1002]	  {};				
		\node[main node] (6)                [below=1.1cm of 1003]   {};
		\node[main node] (7) 			      [below=0.8cm of 5]	  {};				
		\node[main node] (8)                [below=0.8cm of 6]      {};
		\node[draw=none, fill=none] (1004)  [below=0.8cm of 7]        {};  
		\node[draw=none, fill=none] (1005)  [below=0.8cm of 8]        {};  
		\node[draw=none, fill=none] (1006)  [below=0.8cm of 1004]        {};  
		\node[draw=none, fill=none] (1007)  [below=0.8cm of 1005]        {};  
		\node[draw=none, fill=none] (1008)  [below=1.5cm of 1006]        {};  
		\node[draw=none, fill=none] (1009)  [below=1.5cm of 1007]        {};

		\draw (1) edge (3) (2) edge (4) (5) edge (7) (6) edge (8);
		
		\draw (0,-3.5) ellipse (1cm and 2.5cm) (5.5,-3.5) ellipse (1cm and 2.5cm)
		(0,-9.5) ellipse (1cm and 2.5cm) (5.5,-9.5) ellipse (1cm and 2.5cm);
		
		\draw ;
		
		\draw [->] ;	
		
		\draw[dotted] (3) edge (1002) (4) edge (1003)
		(7) edge (1004) (8) edge (1005)
		(1006) edge (1008) (1007) edge (1009)
		(1002) edge (1003) (1002) edge (1005)
		(1004) edge (1003) (1004) edge (1005)
		(1008) edge (1003) (1008) edge (1005)
		(1002) edge (1009) (1004) edge (1009)
		(1008) edge (1009);
		
		\end{tikzpicture} 
		\FloatBarrier
		\caption{Graph, $G$, with $0_2$-invoking subset, $H$, of $V(G)$}
		\label{fig:q-dominating}
	\end{figure}

	Remember that when $H$ perturbs, the edges that have both endpoints in $H$ will have chips travelling along them both ways. So, we can equivalently view these edges as not having any chips travelling along them. For all vertices $h$ in $H$, let $deg_{V(G) \setminus H}(h)$ be the number of vertices in $V(G) \setminus H$ that are adjacent to $h$, and for all vertices $g$ in $V(G) \setminus H$, let $deg_{H}(g)$ be the number of vertices in $H$ that are adjacent to $g$. Thus, when every vertex in $H$ sends a chip to each of its neighbours as a result of the perturbation, the resulting configuration (at step $t=1$) leaves every vertex, $h$, in $H$ with a number of chips equal to $0 - deg_{V(G) \setminus H}(h)$. Every vertex, $g$, in $V(G) \setminus H$ would be left with $0 + deg_{H}(g)$ chips. We know from the definition of CCD that every pair of adjacent vertices in $H$ must be adjacent to the same number of vertices in $V(G) \setminus H$. By transitivity, this will extend to entire connected components within $G|_H$. 
	
	Since the definition of CCD also states that the vertices of $V(G) \setminus H$ follow the same rule with every adjacent pair of vertices being adjacent to the same number of vertices in the complement, we get, by transitivity, that this extends to entire connected components in $G|_{V(G) \setminus H}$. Thus at step 1, each connected component of $G|_H$ will have the same stack size and each connected component of $G|_{V(G) \setminus H}$ will have the same stack size. 
	
	At step 1, every vertex in $H$ has a negative stack size and each vertex in the complement has a positive stack size. So, when the vertices fire at step 1, every vertex in $H$ will receive from each of its neighbours in $V(G) \setminus H$ and will not send to or receive from any vertices in $H$. Likewise, every vertex in $V(G) \setminus H$ will send to each of its neighbours in $H$ and will not send to or receive from any vertices in $V(G) \setminus H$. So for each $h \in H$, we get that 
	
	\begin{align*}
	|h|_2 &= |h|_1 + deg_{V(G) \setminus H}(h)\\ 
	&= -deg_{V(G) \setminus H}(h) + deg_{V(G) \setminus H}(h)\\
	&= 0
	\end{align*} 
	
	\noindent and for all $g \in V(G) \setminus H$,
	
	\begin{align*}
	|g|_2 &= |g|_1 - deg_{H}(g)\\ 
	&= deg_{H}(g) - deg_{H}(g)\\
	&= 0
	\end{align*}
	
	Thus, the fixed configuration is restored in the first two steps.
	
	($\Rightarrow$)	Let $H$ be a perturbation subset of $V(G)$. Suppose $H$ is $0_2$-invoking. This means that if the configuration at step $0$ is the fixed configuration, then so is the configuration at step $2$. This implies that the net effect of two steps of firings on each vertex is $+0$. This implies that for all vertices $h$ in $H$, if $h$ receives a chip from a vertex in $H$ during the firing at step $1$, then $h$ must also send a chip to a vertex in $H$ at step $1$ as well. Every vertex in $H$ will necessarily send a chip to each of its neighbours in $V(G) \setminus H$ as a result of the perturbation (at step 0) and will receive from those same vertices in the firing at step $1$. However following the perturbation, for each connected component $H_i$ in $H$, there must exist some vertex in $H_i$ that has no poorer neighbours in $H_i$. So if any chip is sent from a vertex in $H$ to another vertex in $H$ during the firing at step 1, then there will exist at least one vertex $h_i$ that received a chip from a neighbour in $H$, but did not send a chip to a neighbour in $H$. This implies that $h_i$ will have a positive stack size at step 2, having received more chips in the firing at step 1 than it sent in the initial firing. This, however, contradicts our assumption that $H$ is $0_2$-invoking. Thus, we can conclude that every vertex in a connected component in $G|_H$ has a common stack size after the initial firing. This implies that each vertex belonging to the same connected component in $G|_H$ shares the same number of neighbours in $V(G) \setminus H$. A similar argument will show the result for vertices in $V(G) \setminus H$. Thus, we can conclude that all $0_2$-invoking subsets are $CCD$.
	
\end{proof}

\begin{corollary}\label{cor:iffCCD}
If $H$ is $0_2$-invoking in $G$, then so is $V(G) \setminus H$. 
\end{corollary} 

Note that not all graphs have a proper non-trivial $0_2$-invoking subset. In Figure~\ref{fig:notallgraphs}, we see such a graph. This can be justified by first supposing that $v_2$ were in a such a subset. Now note that either $v_5$ or $v_6$ must be in such a subset. If this subset contains both $v_2$ and $v_6$, then $v_5$ is adjacent to two vertices in $H$ and $v_3$ is adjacent to only one vertex in $H$. So, $v_3$ must be in $H$. Now, $v_4$ must be in $H$ since it is adjacent to only one vertex in $H$ while $v_5$ is adjacent to 3 vertices in $H$. Now, $v_1$ must be in $H$ because $v_4$ is adjacent to 2 vertices in $H$ while $v_3$ is only adjacent to 1. Now, we have finally reached a contradiction as $v_1$ is adjacent to 0 vertices in $V(G) \setminus H$. The other cases can be shown similarly.

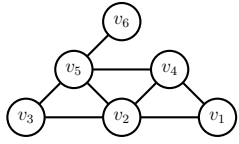
\begin{figure}[H]
	\centering	
	\begin{tikzpicture}[-,-=stealth', auto,node distance=1.5cm,
	thick,scale=0.6, main node/.style={scale=0.6,circle,draw}]
	
	\node[main node] (1)         	         {$v_6$};					% Draws a vertex
	\node[main node] (2)  [below left of=1]        {$v_5$};
	\node[main node] (3)  [below right of=1]        {$v_4$};  
	\node[main node] (4)  [below left of=2]        {$v_3$};					% Draws a vertex
	\node[main node] (5)  [below right of=2]        {$v_2$};
	\node[main node] (6)  [below right of=3]        {$v_1$};	 
	
	\path[every node/.style={font=\sffamily\small}]		%Defines how the edges are drawn
	
	(1) edge node [] {} (2)			
	(2) edge node [] {} (3)
	(2) edge node [] {} (4)
	(2) edge node [] {} (5)
	(3) edge node [] {} (5)
	(3) edge node [] {} (6)
	(4) edge node [] {} (5)
	(5) edge node [] {} (6);

	\end{tikzpicture}
	\caption{Graph with no proper nontrivial $0_2$-invoking subsets}
	\label{fig:notallgraphs}
\end{figure}

We must define some terms to clarify our remaining corollaries, lemmas, and questions regarding $0_2$-invoking subsets.

Given a graph $G$, a \textit{dominating set} is a subset $D$ of $V(G)$ such that every vertex in $V(G)$ is either in $D$ or adjacent to a vertex in $D$. The domination number of a graph $G$,  \textit{$\gamma(G)$}, is the size of the smallest dominating set in $G$. A \textit{minimal dominating set} is a dominating set $M$ such that if any vertex were removed from $M$, then the resulting set would not be dominating.  
Given a graph $G$, an \textit{independent set} is a subset $I$ or $V(G)$ such that no pair of vertices in $I$ are adjacent in $G$.

Note that the graph in Figure~\ref{fig:notallgraphs} has a domination number of 2, while the size of the smallest nontrivial $0_2$-invoking subset of its vertices is 6. This shows us that the $2$-perturbation quiescent number of a graph $G$, $PQ_2(G)$, is not necessarily equal to the domination number $\gamma(G)$.

\begin{corollary}\label{cor:02=>dominating}
	For all graphs $G$, all nontrivial $0_2$-invoking subsets of $V(G)$ are also dominating sets.
\end{corollary}

\begin{proof}
	Let $G$ be a graph and let $H \subseteq V(G)$ be a nontrivial $0_2$-invoking subset. By Theorem~\ref{thm:qdominatingiff}, $H$ is CCD. By the definition of CCD, every vertex in the complement of $H$ must be adjacent to at least one vertex in $H$ unless $H$ is empty. Since $H$ is nontrivial, $H$ is dominating. 
\end{proof}

From \cite{brandstadt:}, an \textit{efficient dominating set}, or \textit{perfect code}, is an independent subset, $A$, of the vertex set of a graph, $G$, such that every vertex in $V(G) \setminus A$ is adjacent to exactly one vertex in $A$.

\begin{corollary}
	Efficient dominating sets (or perfect codes) are CCD and thus $0_2$-invoking.
\end{corollary}

\begin{lemma} \label{lem:minimaldominating}
	Every minimal dominating set of $P_n$, $n \geq 2$, is CCD and thus, $0_2$-invoking.
\end{lemma}

\begin{proof}
	Let $H$ be a minimal dominating set of $P_n$, $n \geq 2$. We will show that $H$ is CCD. Since $H$ is a minimal dominating set, every pair of adjacent vertices in $V(P_n) \setminus H$ are adjacent to exactly one vertex in $H$ each. Since $H$ is a minimal dominating set, every pair of adjacent vertices in $H$ are adjacent to exactly one vertex in $V(P_n) \setminus H$ each. Thus, $H$ is CCD.
\end{proof}

\begin{question}
	Is there a characterization of minimal dominating sets that are also $0_2$-invoking subsets?
\end{question}

If $\gamma(G) = 1$, then there must be a dominating vertex. This vertex is itself a $0_2$-invoking set. If $\gamma(G) = 2$, with dominating set $\{x,y\}$, then the solution is not so simple. We will break the problem into two cases: $x$ not adjacent to $y$, and $x$ adjacent to $y$.
 Suppose first that $x$ and $y$ are not adjacent. For this pair of vertices to also be a $0_2$-invoking set, it must be true that the set $\{x,y\}$ is also complementary component dominant.

\begin{comment}
\begin{figure}[H]
	\centering	
	\begin{tikzpicture}[-,-=stealth', auto,node distance=1.5cm,
	thick,scale=0.6, main node/.style={scale=0.6,circle,draw}]
	
	\node[main node] (1)         	         {$x$};					% Draws a vertex
	\node[draw=none, fill=none] (2)  [left of=1]        {};
	\node[draw=none, fill=none] (3)  [above left=0.8cm of 1]        {};  
	\node[draw=none, fill=none] (4)  [below left=0.8cm of 1]        {};					% Draws a vertex
	\node[draw=none, fill=none] (5)  [right=0.5cm of 1]        {};
	\node[draw=none, fill=none] (6)  [below right=0.8cm of 1]        {};
	\node[draw=none, fill=none] (7)  [above right=0.8cm of 1]    {};
	\node[main node] (8)  [right of=5]        {$y$};  
	\node[draw=none, fill=none] (9)  [right of=8]        {};					% Draws a vertex
	\node[draw=none, fill=none] (10)  [below right=0.8cm of 8]        {};
	\node[draw=none, fill=none] (11)  [above right=0.8cm of 8]        {};

	\draw[dotted]
	(1) edge (2)
	(1) edge (3)
	(1) edge (4)
	(1) edge (5)
	(1) edge (6)
	(1) edge (7)
	(8) edge (5)
	(8) edge (6)
	(8) edge (7)
	(8) edge (9)
	(8) edge (10)
	(8) edge (11);
	
	\end{tikzpicture}
	\caption{Dominating set of 2 vertices}
	\label{fig:notallgraphs2}
\end{figure}

\end{comment}

So, every vertex in a given connected component in $G \setminus \{x,y\}$ must be adjacent to the same number of vertices in $\{x,y\}$ (either 1 or 2). Consider the subset of vertices adjacent to $x$ and not adjacent to $y$, call it $V_x$, and the subset of vertices adjacent to $y$ and not adjacent to $x$, call it $V_y$, and the subset of vertices adjacent to both $x$ and $y$, call it $V_{xy}$. In order for $\{x,y\}$ to be complementary component dominant, it must be true that no edges exist between $V_{xy}$ and $V_x \cup V_y$.

Now, if $x$ and $y$ are adjacent,we must also have an additional rule since $\{x,y\}$ is CCD. If $x$ is adjacent to $y$, then we have the additional rule that $|V_x| = |V_y|$ since both $x$ and $y$ must be adjacent to the same number of vertices. Moving to dominating sets of size 3 or greater appears to be much more difficult.

\begin{comment}

\begin{question}
Is there a characterization of independent dominating sets that are also $0_2$-invoking?	
\end{question}

\end{comment}

In $K_{n,n}$, $n \geq 1$, minimal dominating sets come in two forms: either one vertex from each partition, or an entire partition. In both instances, these sets are CCD and thus, $0_2$-invoking.

In complete multi-partite graphs, minimal dominating sets come in two forms: either one vertex from two different partitions, or an entire partition. The former is not necessarily CCD, while the latter is necessarily CCD. 

\begin{comment}
In a complete multi-partite graph, $K_{n_1,n_2,n_3, \dots,n_k}$, a set composed of one vertex from each part is not necessarily minimally dominating, but is $0_2$-invoking.
\end{comment}

\begin{question}\label{qst:0-invoking}
Is there a graph $G$ such that some subset of $V(G)$ is 0-invoking but not $0_2$-invoking.
\end{question}

\begin{comment}
There does not appear to be such a graph, but this is not based any rigorous approach. 
\end{comment}

\section{Paths}

With a general result characterizing $0_2$-invoking subsets on all graphs, we now focus on paths to show $PQ_2(P_n) = 	\lceil \frac{n}{3} \rceil$, for all $n \geq 1$ (Theorem~\ref{thm:qdominationnumber}), and we determine the number of $0_2$-invoking subsets on a path (Theorem~\ref{thm:qdominatingpath}). 

\subsection{Path Introduction}

Before our results on counting $0_2$-invoking subsets and calculating $PQ_2(P_n)$, we must introduce some definitions and lemmas regarding Diffusion on paths.

Let $G$ be a finite simple undirected graph with vertex set $V(G)$ and edge set $E(G)$. Let $A \subseteq E(G)$. A \textit{graph orientation} of a graph $G$ is a mixed graph obtained from $G$ by choosing an orientation ($x \to y$ or $y \to x$) for each edge $xy$ in $A \subseteq E(G)$. We refer to the edges that are in $E(G)$ \textbackslash $A$ as \textit{flat}. We refer to the assignment of either $x \to y$, $y \to x$, or flat to an edge $xy$ as $xy$'s \textit{edge orientation}. On a path drawn on a horizontal axis, two directed edges in a graph orientation \textit{agree} if they either both point left or both point right.

Let $R$ be a graph orientation of a graph $G$. A \textit{suborientation} $R'$ of $R$ is a graph orientation of some subgraph $G'$ of $G$ such that every edge $xy$ in $G'$ is assigned the same edge orientation as in $R$.

Given two configurations, $C$ and $D$, of a graph $G$, in which the vertices are labelled, $C$ and $D$ are \textit{equal} if $|v|^C = |v|^D$ for all $v \in V(G)$.

In Figure~\ref{fig:parex}, the period length is 2 and the preperiod length is 3.

\begin{lemma}\label{lem:inducego}
	In Diffusion, every configuration induces a graph orientation.
\end{lemma}

\begin{proof}
	Let $G$ be a graph and $C_t$ a configuration on $G$. For all pairs of adjacent vertices $u$, $v$ in $G$ at step $t$, either $u$ gives a chip to $v$, $v$ gives a chip to $u$, the stack sizes of $u$ and $v$ are equal in $C_t$. Let $uv$ be an edge. Assign directions as follows:
	
	\begin{itemize}
		
		\item If $u$ gives a chip to $v$ at time $t$, assign $uv$ the edge orientation $u \to v$.
		
		\item If $v$ gives a chip to $u$ at time $t$, assign $uv$ the edge orientation $v \to u$.
		
		\item If the stack sizes of $u$ and $v$ are equal at time $t$, do not direct the edge $uv$.
		
	\end{itemize}
	
	Thus, a graph orientation on $G$ results.
	
\end{proof}

We say that this graph orientation is \textit{induced} by $C_t$, the configuration of $G$ at time $t$. We see an example of a graph orientation induced by a configuration in Diffusion in Figure~\ref{fig:underlyingorientation}.

\begin{figure} [H] 
	\centering	
	
	\begin{tikzpicture}[-,-=stealth', auto,node distance=1.5cm,
	thick,scale=0.6, main node/.style={scale=0.6,circle,draw}]
	
	\node[main node, label={[red]90:15}] (1) 					   {$v_1$};						% Draws a vertex
	\node[main node, label={[red]90:9}] (2)  [right of=1]        {$v_2$};
	\node[main node, label={[red]90:8}] (3)  [right of=2]        {$v_3$};  
	\node[main node, label={[red]90:2}] (4)  [right of=3]        {$v_4$};						% Draws a vertex
	\node[main node, label={[red]90:12}] (5)  [right of=4]        {$v_5$}; 
	\node[main node, label={[red]90:}] (6) 	[below of=1]				   {$v_1$};						% Draws a vertex
	\node[main node, label={[red]90:}] (7)  [right of=6]        {$v_2$};
	\node[main node, label={[red]90:}] (8)  [right of=7]        {$v_3$};  
	\node[main node, label={[red]90:}] (9)  [right of=8]        {$v_4$};						% Draws a vertex
	\node[main node, label={[red]90:}] (10)  [right of=9]        {$v_5$}; 	
	
	\path (1) edge (2) (2) edge (3) (3) edge (4) (4) edge (5);
	
	\draw [->] (6) edge (7) (7) edge (8) (8) edge (9)
	(10) edge (9) ;

	\end{tikzpicture} 
	\FloatBarrier
	
	\caption{Configuration on $P_5$ and its induced graph orientation.}
	\label{fig:underlyingorientation}

\end{figure}
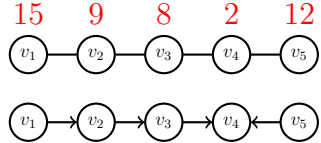

Let $\widebar{Seq(C_0)} = \{ C_t, C_{t+1}, \dots , C_{t+p-1}  \}$ be the ordered set of configurations contained within the period of a configuration sequence $Seq(C_0)$, where $p$ is the length of the shortest period, and the period begins at step $t$.
A configuration $D$ on a graph $G$ is a \textit{period configuration} if $D\in \widebar{Seq}(C)$ for some configuration $C$. A configuration $D$ on a graph $G$ is a \textit{$p_2$-configuration} if $D\in \widebar{Seq}(C)$ for some configuration $C$ and $\widebar{Seq}(C)$ has exactly 2 elements. A \textit{period orientation} is a graph orientation that is induced by a period configuration. A \textit{$p_2$-orientation} is a graph orientation that is induced by a $p_2$-configuration. A \textit{0-orientation} is a graph orientation that is induced by a 0-configuration.

\begin{comment}

\begin{lemma}\label{lem:onlyone}
	Let $G$ be a graph. Up to equivalence, the only fixed configuration on $G$ is the one in which every vertex has $0$ chips. 
\end{lemma}

\begin{proof}
	Let $C$ be a fixed configuration on a graph $G$. We know that, by convention, at least one vertex in $C$ has exactly $0$ chips. Let $R$ be the graph orientation induced by $C$. Suppose, by way of contradiction, that an edge in $R$ is directed. Without loss of generality, let that directed edge be $u \to v$. 
	
	\textbf{Case 1:} There exists a directed cycle $D$ in $R$.
	
	This implies that every vertex in $D$ has both a greater and lesser stack size than every other vertex in $D$. This is a contradiction.
	
	\textbf{Case 2:} There does not exist a directed cycle in $R$.
	
	Let $P$ be a maximal directed path in $R$, containing $u \to v$. Let $x$ be the endpoint of this path with an in-edge but without an out-edge. Note that $x$ may be $v$. This vertex, $x$, is receiving a chip in the initial firing, but it is not sending a chip. Therefore, the stack size of $x$ will change. Thus, $C$ is not a fixed configuration. This is a contradiction.
	
	We can conclude that no directed edges exist in $R$ and thus, all stack sizes in $C$ are equal to $0$.

\end{proof}

\end{comment}

A configuration at step $t$ in a configuration sequence is a \textit{$0$-preposition} if the configuration at step $t+1$ is the  0-configuration. The underlying orientation, $R$, of a configuration is a \textit{$0$-preorientation} if there exists a $0$-preposition which has $R$ as its underlying orientation.  

\begin{lemma}
	Given a graph, $G$, and an orientation $R$, there is at most one configuration which both induces $R$ and is a 0-preposition.
	\label{lem:max1con}	
\end{lemma}

\begin{proof}
	Let $G$ be a graph and $R$ an orientation of $G$. The orientation $R$ dictates the number of chips that each vertex will give and receive at the next firing. Thus, for each vertex $v_k$ in $G$, the stack size of $v_k$ following the next firing is equal to the current stack size of $v_k$ plus the number of edges directed toward $v_k$, $A_{v_k}$, minus the number of edges directed away from $v_k$, $B_{v_k}$. So, if we have that $|v_k| + A_{v_k} - B_{v_k} = 0$, then the stack size of $v_k$ can be determined because it is the only unknown in the equation.    
\end{proof}

\subsection{Results on Paths}

Now, with sufficient background information, we introduce our two main results on paths with Theorems~\ref{thm:qdominationnumber} and \ref{thm:qdominatingpath}.

\begin{theorem}\label{thm:qdominationnumber}
	$PQ_2(P_n) = \lceil \frac{n}{3} \rceil$, $n \geq 1$.
\end{theorem}

\begin{proof}
	We will first prove that $PQ_2(P_n) \geq \lceil \frac{n}{3} \rceil$ and then prove that $PQ_2(P_n) \leq \lceil \frac{n}{3} \rceil$.
	
	($\geq$) From \cite{Chartrand:}, the domination number of a path $P_n$ is $\lceil \frac{n}{3} \rceil$, $n \geq 1$. By Corollary\ref{cor:02=>dominating}, we know that every nontrivial $0_2$-invoking subset of a graph is also a dominating set. Thus, $PQ_2(P_n) \geq \lceil \frac{n}{3} \rceil$.
	
	\begin{comment}
	 Let $H$ be a nontrivial $0_2$-invoking subset in $V(P_n)$. So, if a vertex $v$ in $V(P_n) \setminus H$ does not have any neighbours in $H$, then neither do the neighbours of $v$. This implies that either $H$ is trivial or every vertex in $V(P_n) \setminus H$ is adjacent to a vertex in $H$. Since we have supposed that $H$ is nontrivial, we conclude that $H$ is a dominating set.
	\end{comment}
	
	($\leq$) By Lemma~\ref{lem:minimaldominating}, we know that every minimal dominating set of a path, $P_n$, is also a $0_2$-invoking subset of $V(P_n)$. Therefore, $PQ_2(P_n) \leq \lceil \frac{n}{3} \rceil$.

\end{proof}

Let $J_n$ represent the number of $0_2$-invoking subsets that exist on $P_n$. We will now count all $0_2$-invoking subsets on a path with $n \geq 2$ vertices.

We will label $P_n$ to have vertices $v_1$, $v_2$, $v_3$, $\dots$, $v_n$.

\begin{lemma}\label{lem:vnvn-1}
	Let $H \subset V(P_n) = \{v_1,v_2, \dots ,v_{n-1},v_n\}$, $n \geq 2$ be proper, non-trivial, and $0_2$-invoking. Then $v_n \in H$ if and only if $v_{n-1} \in V(G) \setminus H$.  	
\end{lemma}

\begin{proof}
	Let $H \subset V(P_n) = \{v_1,v_2, \dots v_{n-1},v_n\}$ be $0_2$-invoking, proper and nontrivial. 
	($\Rightarrow$) Suppose first that $v_n \in H$. We know that $v_{n-1}$ is the only neighbour of $v_n$ in $P_n$. If $v_{n-1} \in H$, then $v_n$ would be adjacent to 0 vertices in $V(P_n) \setminus V(H)$ and thus, since $H$ is $0_2$-invoking, every vertex in the same connected component as $v_n$ in $G|_H$ would be adjacent to 0 vertices in $V(G) \setminus H$. Since $P_n$ is connected, this implies that $H$ is not a proper subset of $V(P_n)$ which is a contradiction. Thus, if $v_n \in H$, then $v_{n-1} \in V(G) \setminus H$.
	
	\noindent($\Leftarrow$) Suppose now that $v_{n-1} \in V(G) \setminus H$. Then if $v_n \in V(G) \setminus H$, it would be adjacent to 0 vertices in $H$ and thus, since $H$ is $0_2$-invoking, every vertex in the same connected component as $v_n$ in $G|_H$ would be adjacent to 0 vertices in $H$. Since $P_n$ is connected, this implies that $H$ is the trivial subset of $V(G)$ which is a contradiction. Thus, if $v_{n-1} \in V(G) \setminus H$, then $v_n \in H$.   
\end{proof}

\begin{theorem}\label{thm:qdominatingpath}
	$J_{n} = J_{n-1} + J_{n-2} - 2$, for $n \geq 3$, with $J_1 = 2$ and $J_2 = 4$.
\end{theorem}

\begin{proof}
	Note first that we are including the trivial and improper cases, so as to count every $0_2$-invoking set on $P_n$. We begin with the initial values. The path with only one vertex cannot send chips because it has no edges. Thus, whether the lone vertex is in the perturbation subset or not, the chosen set is $0_2$-invoking. So, $P_1$ has two $0_2$-invoking subsets: $\emptyset$ and $V(P_1)$. On $P_2$, a perturbation of any subgraph will return to the fixed configuration after another step. Thus, $P_2$ has four $0_2$-invoking subsets.	
	
	Trivially, the empty subset and the entire vertex set are $0_2$-invoking in $P_n$. We will take note of this and move forward counting the $0_2$-invoking subgraphs that are both nonempty and have nonempty complement. 
	
	We will view the problem of partitioning the vertices of a path into $H$ and $V(G) \setminus H$ as a colouring problem, colouring the vertices of $P_n, n \geq 2,$ red if they are in $H$ and blue if they are in $V(G) \setminus H$. Suppose we have $P_n$ coloured in such a way that $H$ (and thus, also $V(G) \setminus H$) is a $0_2$-invoking subset. Suppose also that at least one vertex is red and at least one vertex is blue. We will now count all such possible colourings and we will refer to these as \textbf{$0_2$-invoking colourings}.
	
	By Lemma~\ref{lem:vnvn-1}, we know that $v_n$ and $v_{n-1}$ must be different colours, see Figure~\ref{fig:coloursFibonacci2}.

	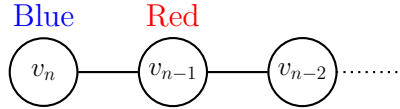
\begin{figure}[H]
		\centering

		\begin{tikzpicture}[-,-=stealth', auto,node distance=1.5cm,
		thick,scale=0.6, main node/.style={scale=0.6,circle,draw, minimum size=1.5cm, font=\sffamily\Large\bfseries}]

		\node[main node, label={[blue]90:Blue}]   (1)                                          {$v_{n}$};
		\node[main node, label={[red]90:Red}]     (2) 	    [right=0.8cm of 1]    	         {$v_{n-1}$};
		\node[main node, label={[blue]90:}]   (3)         [right=0.8cm of 2]               {$v_{n-2}$};
		\node[draw=none, fill=none] (1000)                    [right=0.8cm of 3]               {};

		\draw  (1) edge (2) (2) edge (3);	
		
		\draw[dotted] (1000) edge (3);
		
		\end{tikzpicture} 
		\caption{The two vertices, $v_{n}$ and $v_{n-1}$ must have different colours since they are adjacent to different numbers of blue vertices.}
		\label{fig:coloursFibonacci2}
		\FloatBarrier
	\end{figure}
	
	Suppose $P_n$ has a $0_2$-invoking colouring. Look at the colours assigned to the final three vertices: $v_n$, $v_{n-1}$, and $v_{n-2}$. By Lemma~\ref{lem:vnvn-1}, we are able to exclude some possible colourings of these final three vertices. See Table~\ref{tab:colouringstable1}.

	\begin{table}[H] 
		\centering
		\begin{tabular}{|c|c|}
			\hline
			Colouring of $v_nv_{n-1}v_{n-2}$ & $\#$ of $0_2$-invoking colourings\\
			\hline
			$RRR \cdot \cdot \cdot$  & 0\\
			\hline
			$RRB \cdot \cdot \cdot$  & 0\\
			\hline
			$RBR \cdot \cdot \cdot$  & ?\\
			\hline
			$RBB \cdot \cdot \cdot$  & ?\\
			\hline
			$BRR \cdot \cdot \cdot$  & ?\\
			\hline
			$BRB \cdot \cdot \cdot$  & ?\\
			\hline
			$BBR \cdot \cdot \cdot$  & 0\\
			\hline
			$BBB \cdot \cdot \cdot$  & 0\\
			\hline	
		\end{tabular}
		\caption{Colourings of the last three vertices of a path}
		\label{tab:colouringstable1}
	\end{table}
	
	\noindent For the four remaining possible colourings of $v_n$, $v_{n-1}$, and $v_{n-2}$, we will develop a recurrence relation, building on values from smaller paths.
	
	\noindent \textbf{Case 1:} Suppose $P_n$ has a $0_2$-invoking colouring in which $v_n$ is red, $v_{n-1}$ is blue, and $v_{n-2}$ is red. In Theorem~\ref{thm:qdominatingiff}, it is shown that a subset of the vertices of a graph is $0_2$-invoking if and only if it is CCD and in Corollary~\ref{cor:iffCCD}, it is shown that a subset of the vertices of a graph is $0_2$-invoking if and only if its complement is $0_2$-invoking as well. So a colouring is $0_2$-invoking if and only if the two colour classes are CCD. Clearly if we were to remove $v_n$ from this colouring, yielding a colouring on $P_{n-1}$, the resulting colouring would be CCD since $v_{n-1}$, the only vertex which was adjacent to $v_n$, is not adjacent to any other blue vertices. Thus, for every $0_2$-invoking colouring of $P_n$ in which $v_n$ is red, $v_{n-1}$ is blue, and $v_{n-2}$ is red, there exists exactly one $0_2$-invoking colouring of $P_{n-1}$ in which $v_{n-1}$ is blue and $v_{n-2}$ is red. Since there is no fundamental difference between the colours red and blue, and the final two vertices must have opposing colours by Lemma~\ref{lem:vnvn-1}, the number of $0_2$-invoking colourings of $P_{n-1}$ in which $v_{n-1}$ is blue and $v_{n-2}$ is red is equal to half of the total number of $0_2$-invoking colourings of $P_{n-1}$. The number of $0_2$-invoking colourings of $P_{n-1}$ is equal to $J_{n-1} - 2$ (remembering to account for the improper and trivial cases which are $0_2$-invoking but are not defined to be $0_2$-invoking colourings). Thus, the number of $0_2$-invoking colourings of $P_n$ in which $v_n$ is red, $v_{n-1}$ is blue and $v_{n-2}$ is red is equal to $\frac{1}{2}(J_{n-1} - 2) = \frac{1}{2}J_{n-1} - 1$.
	
	\noindent \textbf{Case 2:} Suppose $P_n$ has a $0_2$-invoking colouring in which $v_n$ is blue, $v_{n-1}$ is red, and $v_{n-2}$ is blue. Since there is no fundamental difference between the colours red and blue, we know that there are also $\frac{1}{2}J_{n-1} - 1$ $0_2$-invoking colourings of this form on $P_n$.

	\begin{table}[H] 
		\centering
		\begin{tabular}{|c|c|}
			\hline
			Colouring of $v_nv_{n-1}v_{n-2}$ \dots & $\#$ of $0_2$-invoking colourings\\
			\hline
			$RRR \cdot \cdot \cdot$  & 0\\
			\hline
			$RRB \cdot \cdot \cdot$  & 0\\
			\hline
			$RBR \cdot \cdot \cdot$  & $\frac{1}{2}J_{n-1} - 1$\\
			\hline
			$RBB \cdot \cdot \cdot$  & \\
			\hline
			$BRR \cdot \cdot \cdot$  & \\
			\hline
			$BRB \cdot \cdot \cdot$  & $\frac{1}{2}J_{n-1} - 1$\\
			\hline
			$BBR \cdot \cdot \cdot$  & 0\\
			\hline
			$BBB \cdot \cdot \cdot$  & 0\\
			\hline	
		\end{tabular}
		\caption{Colourings of the last three vertices of a path}
		\label{tab:colouringstable2}
	\end{table}
	
	\noindent \textbf{Case 3:} Suppose $P_n$ has a $0_2$-invoking colouring in which $v_n$ is red, and both $v_{n-1}$ and $v_{n-2}$ are blue. By Theorem~\ref{thm:qdominatingiff}, we know that both colour sets are CCD. So we know that $v_{n-2}$ must be adjacent to a red vertex. Thus, $v_{n-3}$ is red. However, we do not know whether $v_{n-4}$ is red or blue. We have no knowledge of the remainder of the colours except that both colour sets are CCD. If we were to remove $v_n$ and $v_{n-1}$, the resulting colouring of $P_{n-2}$ would be $0_2$-invoking because the only vertex adjacent to either of these vertices is $v_{n-2}$, and in the resulting colouring of $P_{n-2}$, $v_{n-2}$ is not adjacent to any other blue vertices. Thus, the number of $0_2$-invoking colourings of $P_n$ in which $v_n$ is red, and both $v_{n-1}$ and $v_{n-2}$ are blue is equal to the number of $0_2$-invoking colourings of $P_{n-2}$ in which $v_{n-2}$ is blue and $v_{n-3}$ is red. Since there is no fundamental difference between the colours red and blue, and the final two vertices must have opposing colours by Lemma~\ref{lem:vnvn-1}, the number of $0_2$-invoking colourings of $P_{n-2}$ in which $v_{n-2}$ is blue and $v_{n-3}$ is red is equal to half of the total number of $0_2$-invoking colourings of $P_{n-2}$. The number of $0_2$-invoking colourings of $P_{n-2}$ is equal to $J_{n-2} - 2$ (remembering to account for the improper and trivial cases which are $0_2$-invoking but are not defined to be $0_2$-invoking colourings). Thus, the number of $0_2$-invoking colourings of $P_n$ in which $v_n$ is red, and both $v_{n-1}$ and $v_{n-2}$ are blue is equal to $\frac{1}{2}(J_{n-2} - 2) = \frac{1}{2}J_{n-2} - 1$.
	
	\noindent \textbf{Case 4:} Suppose $P_n$ has a $0_2$-invoking colouring in which $v_n$ is blue, and both $v_{n-1}$ and $v_{n-2}$ are blue. Since there is no fundamental difference between the colours red and blue, we know that there are also $\frac{1}{2}J_{n-2} - 1$ $0_2$-invoking colourings of this form on $P_n$.

	\begin{table}[H] 
		\centering
		\begin{tabular}{|c|c|}
			\hline
			Colouring of $v_nv_{n-1}v_{n-2}$ \dots & $\#$ of $0_2$-invoking colourings\\
			\hline
			$RRR \cdot \cdot \cdot$  & 0\\
			\hline
			$RRB \cdot \cdot \cdot$  & 0\\
			\hline
			$RBR \cdot \cdot \cdot$  & $\frac{1}{2}J_{n-1} - 1$\\
			\hline
			$RBB \cdot \cdot \cdot$  & $\frac{1}{2}J_{n-2} - 1$\\
			\hline
			$BRR \cdot \cdot \cdot$  & $\frac{1}{2}J_{n-2} - 1$\\
			\hline
			$BRB \cdot \cdot \cdot$  & $\frac{1}{2}J_{n-1} - 1$\\
			\hline
			$BBR \cdot \cdot \cdot$  & 0\\
			\hline
			$BBB \cdot \cdot \cdot$  & 0\\
			\hline	
		\end{tabular}
		\caption{Colourings of the last three vertices of a path}
		\label{tab:colouringstable3}
	\end{table}

	\noindent So $J_{n} - 2 = J_{n-1} - 2 + J_{n-2} - 2$. Therefore $J_{n} = J_{n-1} + J_{n-2} - 2$. 
\end{proof}

\begin{corollary}
	Let $F_i$ be the $i^{th}$ Fibonacci number with $F_0 = 0$, $F_1 = 1$, and $F_i = F_{i-1} + F_{i-2}$. Then $J_{k+1} = 2(F_{k}+1)$.
\end{corollary}	     

\begin{proof}	     
	Note that 
	
	$J_1 = 2(F_0+1) = 2$;
	
	$J_2 = 2(F_1+1) = 4$; 
	
	$J_3 = 2(F_2+1 )= 4$. 
	
	Assume for $2 \leq i \leq k$ that $J_i = 2(F_{i-1}+1)$. Then
	\begin{align*}
	J_{k+1} &= J_{k}+J_{k-1}-2\\ 
	&= 2(F_{k-1}+1+F_{k-2}+1) - 2\\
	&=2(F_{k-1}+F_{k-2} + 1)\\
	&=2(F_{k}+1).\\
	\end{align*}
\end{proof}

\section{Conclusion}

The results in this paper revolve around the broad question ''How can we categorize those perturbation subsets that lead to the 0-configuration being eventually restored after some amount of steps?" The definition of a $0_2$-invoking subset and the characterization of all $0_2$-invoking subsets as CCD (Theorem~\ref{thm:qdominatingiff}) begins to paint this picture for us. This naturally gives rise to Question~\ref{qst:0-invoking}, which we restate here: 

\textbf{Question~\ref{qst:0-invoking}:} Is there a graph such that some subset of its vertex set is 0-invoking but not $0_2$-invoking.

In \cite{mullen:}, we develop a method of counting and characterizing every configuration on a path that will lead to the 0-configuration in the next step. This serves as the first step in what may be a method of answering Question~\ref{qst:0-invoking}, because characterizing all of the configuration sequences that can eventually lead to the 0-configuration would inevitably determine whether or not any such firing set exists (at least on a path).

In its most general form, a perturbation is just some kind of disruption to the stack sizes of a configuration. In this paper, we analyzed when a perturbation of the $0$-configuration eventually returned to the $0$-configuration. However, a more general question would be ``Which configurations can return after being perturbed?" Does it matter how many vertices are in a perturbation set when answering this question?

\begin{comment}
\nocite{*}
\bibliographystyle{abbrvnat}
% use the following instead if you encounter problems 
%\bibliographystyle{alpha}
\bibliography{QuantumPaper}

\begin{thebibliography}{99}



\bibitem{brandstadt:} A. Brandst{\"a}dt, A. Leitert, and D. Rautenbach, 2012, Efficient dominating and edge dominating sets and hypergraphs, \textit{Algorithms and Computation}, \textit{Lecture Notes in Comput. Sci.}, Springer and Heidelberg, 7676, 267-277.

\bibitem{carlotti:} A. Carlotti, R. Herrman, 2018, Uniform bounds for non-negativity of the diffusion game, arXiv:1805.05932v1. 

\bibitem{Chartrand:} G. Chartrand and P. Zhang, Introduction to Graph Theory, McGraw Hill, 2005.

\bibitem{duffy:} C. Duffy, T.F. Lidbetter, M.E. Messinger, R.J. Nowakowski, 2018, A Variation on Chip-Firing: the diffusion game, \textit{Discrete Mathematics \& Theoretical Computer Science}, 20, \#4.

\bibitem{long:} J. Long and B. Narayanan, 2019, Diffusion On Graphs Is Eventually Periodic, \textit{Journal of Combinatorics}, 10, no.2, 235-241.

\bibitem{mullen:} T. Mullen, \textit{On Variants of Diffusion}, PhD thesis, Dalhousie University, 2020.

\end{thebibliography}
\label{sec:biblio}
\end{comment}

\end{document}